\documentclass{SK}

\begin{document}

\title{Centrally endo-AIP modules}

\author[Shiv Kumar, Ashok Ji Gupta]{Shiv Kumar\affil{1}, Ashok Ji Gupta\affil{2}\comma\corrauth}

\address{\affilnum{1}\ Indian Institute of Technology, Banaras Hindu University\\ Varanasi-221005, Uttar Pradesh, INDIA\\
\affilnum{2}\ Indian Institute of Technology, Banaras Hindu University\\ Varanasi-221005, Uttar Pradesh, INDIA}

\emails{{\tt shivkumar.rs.mat17@itbhu.ac.in}\ (Shiv Kumar), {\tt agupta.apm@itbhu.ac.in}\ (Ashok Ji Gupta)}

\begin{abstract}
In this paper, we introduce the concept of centrally endo-AIP modules. We call a module $M$ centrally endo-AIP, if the left annihilator of any fully invariant submodule $N$ of $M$ in the endomorphism ring $S=End_R(M)$ is a centrally s-unital ideal of $S$. We discuss some properties of centrally endo-AIP modules. We also study the endomorphism ring of centrally endo-AIP modules and characterize quasi-Baer modules in terms of centrally endo-AIP modules.
\end{abstract}


\maketitle

\section{Introduction }
Throughout this paper, we consider all the rings are associative with unity and all modules are right unitary, unless otherwise stated. I. Kaplansky \cite{IK} called a ring $R$, Baer (quasi-Baer) if the right annihilator of any subset (ideal) of $R$ is generated as a right ideal by an idempotent element of $R$ and introduced these notions to study some properties of von Neumann regular algebras, $AW^*$-algebras and $C^*$-algebras. Many researchers focused on the Baer rings because it has a root in functional analysis and has a close link to $C^*$-algebras as well as to von Neumann algebras. When the right (left) annihilator of any element of $R$ in $R$ is a direct summand of $R$, then $R$ is called a right (left) Rickart ring or right (left) PP ring \cite{CE}. Birkenmeier et al. \cite{Birken} gave another generalization of a Baer ring, which is known as a principally quasi-Baer ring (p.q.-Baer) and called a ring $R$ principally quasi-Baer (p.q.-Baer) if right annihilator of every principal ideal is generated by an idempotent element of $R$. An ideal $I$ of a ring $R$ is called right (left) s-unital ideal of $R$ if for each $a\in I$, $ax=a$ (resp. $xa=a$) for some element $x\in I$ (see \cite{ZR}). Also an ideal is said to be a centrally s-unital ideal of $R$ if for every $a\in I$, there exists a central element $z\in I$ such that $az=a$. Furthermore, a submodule $N$ of a right $R$-module $M$ is called a pure submodule of $M$, if the sequences $0\rightarrow N\rightarrow M$ and $0\rightarrow K\otimes N\rightarrow K\otimes M$ remain exact for every left $R$-module $K$ \cite{PM}. The condition for a right $R$-module $M$ to be flat is that whenever $0\rightarrow N_1\rightarrow N_2$ is exact for left $R$-modules $N_1$ and $N_2$ then $0\rightarrow M\otimes N_1\rightarrow M\otimes N_2$ is also exact. By using the concept of s-unital ideal, Liu and Zhao \cite{ZR}, defined a generalised structure of $PP$ rings and p.q.-Baer rings. According to them, a ring $R$ is AIP (APP) if the right (left) annihilator of every ideal (resp. principal ideal) $I$ of $R$ is pure right (left) ideal of $R$ or the right (left) annihilator of any ideal (resp. principal ideal) $I$ of $R$ is a left (right) s-unital ideal of $R$. The class of AIP rings properly contains the class of PP rings and p.q.-Baer rings, which are contained in the class of APP rings. They also introduced centrally AIP rings which contains PP rings and abelian p.q.-Baer rings (see \cite{AAK}). A ring $R$ is called a centrally left AIP-ring if left annihilator of every ideal $I$ of $R$, is a centrally s-unital ideal of $R$. 
P.A. Dana and A. Moussavi \cite{DM}, introduced the module theoretical notion AIP and APP rings as endo-AIP and endo-APP modules. A module $M$ is said to be an endo-AIP (endo-APP) if the left annihilator of every fully invariant (resp. cyclic) submodule of $M$ is a right s-unital ideal of $S$ or a pure left ideal of $S$.\\

In this paper, we introduce module theoretical notion of centrally AIP rings as centrally endo-AIP modules. An $R$-module $M$ is said to be a centrally endo-AIP module if left annihilator of every fully invariant submodule of $M$ in $S=End(M)$ is a centrally s-unital ideal of $S$. Every abelian Rickart module is a centrally endo-AIP module and every centrally endo-AIP module is an endo-AIP module (see, Proposition \ref{CEE1.1}). We show that centrally endo-AIP module is closed under direct summand (see Proposition \ref{DS1.1}). In general direct sum of centrally endo-AIP modules need not be centrally endo-AIP. We find the conditions for which the direct sum of centrally endo-AIP modules is centrally endo-AIP (see Proposition \ref{P2.17} and Theorem \ref{T2.18}). We also prove that every projective $R$-module is centrally endo-AIP if and only if $R$ is a centrally AIP ring (see Theorem \ref{Free1.1}).\\

In section 3, we study the endomorphism ring of centrally endo-AIP modules. The endomorphism ring of centrally endo-AIP module is centrally AIP ring or semiprime ring (see Proposition \ref{P3.1}). Further, we show that for a locally quasi-retractable module $M$, the ring of endomorphisms $S=End_R(M)$ is centrally AIP iff $M$ is centrally endo-AIP module (see Proposition \ref{LPQ1.1}). Also, we prove that the endomorphism ring $S=End_R(M)$ of a centrally endo-AIP module $M$ is a quasi-Baer if $S$ has a finite left uniform dimension (see Proposition \ref{UD1.1})\\

We fix the notations $\subseteq $, $\leq$, $\leq ^{\oplus}$, $\leq ^{e}$, $\trianglelefteq$ and $\trianglelefteq ^p$ to denote a subset, a submodule, a direct summand, an essential submodule, a fully invariant submodule (or an ideal) and a projection invariant submodule respectively. For an $R$-module $M$ with endomorphism ring $S=End_R(M)$, $r_M(I)$ (where $I$ is a left ideal of $S$) and $l_S(N)$ (where $N\leq M$) will denote the annihilator of $I$ in $M$ and annihilator of $N$ in $S$ respectively. An idempotent element $e^2=e\in S$ is said to be left (right) semicentral if for every $x\in S$ $exe=xe$ ($exe=ex$). $\mathbb{S}_l(S)$ ($\mathbb{S}_r(S)$) denotes the set of all left (right) semicentral idempotent elements of $S$. By a regular ring, we always mean a von Neumann regular and $T_n(R)$ stands for $n\times n$ triangular matrix ring over a ring $R$. Before proceed to the main section, we recall some definitions and results which will be helpful to the clarity of further results.\\
We recall the following definitions from \cite{ORM}.
\begin{definition}
	Let $M$ be an $R$-module with $S=End_R(M)$.
	\begin{enumerate}
		\item[(i)] $M$ is said to be reduced if for each $\phi \in S$ and $m\in M$, $\phi(m)=0$ implies $Im(\phi)\cap Sm=0$ . Equivalently, $M$ is a reduced module if $\phi^2(m)=0$ implies $\phi S(m)=0$.
		\item[(ii)] $M$ is called a rigid module if for every $\psi \in S$ and $m\in M$, $\psi^2(m)=0$ implies $\psi (m)=0$ . Equivalently $Ker(\psi)\cap Im(\psi)=0$ for every $\psi \in S$.
		\item[(iii)] A ring $R$ is said to be abelian if every idempotent element of $R$ is central. Further, a module $M$ is called abelian if its endomorphism ring $S$ is abelian. In other words, $M$ is an abelian module if $\psi e(m)=e\psi (m)$ for every $m\in M$ where $\psi \in S$ and $e^2=e\in S$.
		\item[(iv)] $M$ is said to be symmetric, if $\phi \psi (m)=0$  implies $\psi \phi (m)=0$ for every $\phi, \psi \in S$ and $m\in M$. 
		\item[(v)] $M$ is known as semicommutative \cite{Nazim}, if $\psi (m)=0$ implies $\psi S(m)=0$ for every $\psi \in S$ and $m\in M$.
	\end{enumerate}
\end{definition} 
Reduced modules, rigid modules, symmetric modules and semicommutative modules are abelian, for details see \cite{ORM}.
\begin{lemma}\label{RL1.1}
	(Theorem 2.25, \cite{ORM}), The following statements are equivalent for a Rickart module $M$:
	\begin{enumerate}
		\item[(i)] $M$ is an abelian module;
		\item[(ii)] $M$ is a reduced module;
		\item[(iii)] $M$ is a rigid module;
		\item[(iv)] $M$ is a semicommutative module;
		\item[(v)] $M$ is a symmetric module.
	\end{enumerate}
\end{lemma}
\begin{definition}\label{Baer}
	Let $M$ be an $R$-module and $S=End_R(M)$. Then
	\begin{enumerate}
		\item[(i)] $M$ is said to be Baer (quasi-Baer) module \cite{RR}, if for every submodule (fully invariant submodule) $N$ of $M$, $l_S(N)$ is a direct summand of $S$. Further, $M$ is called principally quasi-Baer module \cite{GL}, if for every cyclic submodule $P$ of $M$, $l_S(P)$ is a direct summand of $S$.
		\item[(ii)]  A module $M$ is called Rickart if for every endomorphism $\phi \in S$, $Ker(\phi)$ is a direct summand of $M$.
		\item[(iii)] A module $M$ is called retractable if $Hom(M,N)\neq 0$, for all $0\neq N\leq M$. Equivalently, $M$ is retractable module if there exists $0\neq \psi \in S=End_R(M)$ with $Im(\psi)\subseteq N$ for every $N\leq M$.
	\end{enumerate}
\end{definition}
It is easy to see from  definition \ref{Baer} that the following heierarchy is true, \\Baer module $\Rightarrow$ quasi-Baer module $\Rightarrow$ principally quasi-Baer module

\section{Centrally endo-AIP modules}

\begin{definition}\label{D2.1}
	An $R$-module $M$ is called a centrally endo-AIP module, if the left annihilator of any fully invariant submodule of $M$ in  $S=End_R(M)$ is a centrally s-unital ideal of $S$. Equivalently, for every $N \trianglelefteq M$ and for each $\phi \in l_S(N)$ there exists a central element $\psi \in l_S(N)$ such that $\phi \psi =\phi=\psi\phi$. Although, a ring $R$ is said to be a right centrally AIP if $R_R$ is a centrally endo-AIP $R$-module.
\end{definition}
The following proposition provides a rich source of examples of centrally endo-AIP modules.
\begin{proposition}\label{CEE1.1}
	Let $M$ be an $R$-module and $S=End_R(M)$. Consider the following statements:
	\begin{enumerate}
		\item[(i)] $M$ is an abelian Rickart module;
		\item[(ii)] $M$ is a centrally endo-AIP module;
		\item[(iii)] $M$ is an endo-AIP module.
	\end{enumerate}
	Then, $(i)\Rightarrow (ii)\Rightarrow (iii)$ while converse of these implications need not be true.
\end{proposition}
\begin{proof}
	$(i)\Rightarrow (ii)$  Let $N$ be a fully invariant submodule of $M$ and $f \in l_S(N)$ be arbitrary. Then, $f (N)=0$ which implies $N\subseteq Ker(f)$. Since $M$ is a Rickart module, $Ker(f)$ is a direct summand of $M$. So for some $e^2=e\in S$, $N\subseteq Ker(f)=eM$. Thus, $(1-e)N=0$ and $f e=0$. Therefore, $(1-e)\in l_S(N)$ and $f(1-e)=f$. By hypothesis $M$ is an abelian module, so every idempotent of $S$ is central. Therefore, $(1-e)$ is a central idempotent element of $l_S(N)$ such that $f(1-e)=f$. Hence, $M$ is a centrally endo-AIP module.\\
	$(ii)\Rightarrow (iii)$ Let $N$ be a fully invariant submodule of $M$ and $\psi \in l_S(N)$, then $\psi N=0$ which implies that $N\subseteq r_M(\psi)$. Since $M$ is a centrally endo-AIP module, so there exists a central element $\phi \in l_S(N)$ such that $\psi \phi = \psi$.  Therefore, $l_S(N)$ is a right s-unital ideal of $S$. Hence, $M$ is an endo-AIP module.\\
	$(ii)\nRightarrow (i)$ Let $P$ be a local prime ring which is not a domain and $J$ be the jacobson radical of $P$. Let $R=\{(x, \bar{y})\: | \: x\in J, \; \bar{y} \in \bigoplus _{n=1}^{\infty} R_n \}$, where $R_n=P/J$ for each $n$, $\bar{y}=(\bar{y}_n)_{n=1}^{\infty}$ and $\bar{y}_n=y_n +J\in R_n$. It is clear from (Example 2.8, \cite{AAK}), the ring $R$ is centrally AIP-ring which is neither abelian ring nor Rickart ring. Therefore, $R_R$ is a centrally endo-AIP $R$-module while $R_R$ is neither abelian $R$-module nor Rickart $R$-module.\\
	$(iii)\nRightarrow (ii)$ Let $R=\begin{pmatrix}
		\Pi _{i=1}^{\infty} \mathbb{F}_i &	\bigoplus _{i=1}^{\infty} \mathbb{F}_i \\ 	\bigoplus _{i=1}^{\infty} \mathbb{F}_i & \langle 	\bigoplus _{i=1}^{\infty} \mathbb{F}_i , 1 \rangle
	\end{pmatrix}$ and $M=R$, where $\mathbb{F}$ is any field and $\mathbb{F}_i = \mathbb{F}$ for $i=1,2,3,...$.	It is clear from (Example 1.6, \cite{Birken}), that the ring $R$ is semiprime left $PP$-ring, so it is semiprime left $AIP$-ring. Thus, $M$ is an endo-AIP module. Now from (Example 2.11, \cite{AAK}), $R$ is not a centrally AIP-ring. Therefore, $M$ is not a centrally endo-AIP $R$-module.\\
	$(iii)\nRightarrow (i)$  Let $R=T_n(F)$ be an upper triangular matrix ring over $F$, where $F$ is a domain which is not a division ring. Then, by (Example 2.6, \cite{DM}) $R_R$ is an endo-AIP $R$-module but not a Rickart $R$-module (see Example 2.9, \cite{GL}).
\end{proof}
\begin{corollary}\label{RCE1.1}
	A Rickart module $M$ is a centrally endo-AIP module, if $M$ satisfies any one of the following:
	\begin{enumerate}
		\item[(i)] $M$ is reduced.
		\item[(ii)] $M$ is rigid.
		\item[(iii)] $M$ is abelian.
		\item[(iv)] $M$ is semicommutative.
		\item[(v)] $M$ is symmetric.
	\end{enumerate}
\end{corollary}
\begin{proof}
	It follows from lemma \ref{RL1.1} and from Proposition \ref{CEE1.1}.
\end{proof}
According to Liu and Ouyang (Definition 3.2, \cite{QBT}), a right $R$-module $M$ is said to have insertion of factor property (IFP) if $r_M(\psi)\trianglelefteq M$ for all $\psi \in S=End_R(M)$. Equivalently, $l_S(m)$ is an ideal of $S$ for every $m\in M$
\begin{proposition}\label{IFP1.1}
	Let $M$ be a module with insertion of factor property and $S=End_R(M)$. Then, the following statements are equivalent:
	\begin{enumerate}
		\item[(i)] $M$ is a centrally endo-AIP module;
		\item[(ii)] $M$ is an endo-AIP module;
		\item[(iii)] $M$ is an endo-APP module.
	\end{enumerate} 
\end{proposition}
\begin{proof}
	$(i)\Rightarrow (ii)$ It is clear from Proposition \ref{CEE1.1}.\\
	$(ii)\Rightarrow (i)$  Suppose that $M$ has IFP and $N\trianglelefteq M$.  Let $\phi \in l_S(N)$ be arbitrary. Then, $\phi (N)=0$ which implies $N \subseteq r_M(\phi)$. By assumption $M$ is an endo-AIP module with IFP so from (Proposition 2.10, \cite{PM}), $M$ is a Rickart module. Therefore, $N\subseteq r_M(\phi)=e(M)$ for some idempotent element $e^2 =e\in S$. Thus, $(1-e)N=0$ and $\phi e=0$. So, we have $(1-e)\in l_S(N)$ and $\phi (1-e)=\phi$. Further, $M$ has IFP property, so by (Proposition 3.4, \cite{QBT}) $S$ is an abelian ring. Therefore, the idempotent element $(1-e)$ is central. Hence, $M$ is a centrally endo-AIP module.\\
	$(ii)\Leftrightarrow (iii)$ It follows from (Proposition 4.3, \cite{DM}).	
\end{proof}

In Proposition \ref{IFP1.1}, the insertion of factor property (IFP) is not superfluous. We justify it by the following example.
\begin{example}\label{E2.7}
	Let $R=T_n(\mathbb{F})$ and $M=R_R$, where $\mathbb{F}$ is a domain which is not a division ring. Then, by (Theorem 3.5, \cite{AAK}) $R$ is a right centrally AIP ring. Therefore, $M$ is a centrally endo-AIP module but not a Rickart module (see Example 2.9, \cite{GL}). Thus, $M$ is an endo-AIP module. Let $T_{ij} \in T_2(\mathbb{F})$ where $T_{ij}$ with $1$ at $(i , j)$-position and $0$ elsewhere for every $i , j=1,2$. Then, $T_{11}T_{22}=0$ but $T_{11}T_{12}T_{22}=0$. So, $R$ does not have IFP.  Therefore, $M$ does not satisfy insertion of factor property.
\end{example}

\begin{proposition}\label{DS1.1}
	Direct summand of centrally endo-AIP module is a centrally endo-AIP.
\end{proposition}
\begin{proof}
	Let $M$ be a centrally endo-AIP module with  $S=End_R(M)$ and $N\leq^{\oplus} M$. Then, for some idempotent $e^2=e\in S$, $N=eM$ and $T=End_R(N)=eSe$. Let $K$ be a fully invariant submodule of $N$. Clearly, $SK$ is also a fully invariant submodule of $M$. Suppose $\psi \in l_T(K)$, then there exists some $\phi \in S$ such that $\psi =e\phi e$. Now, $e\phi e (SK)=e\phi (eS(eK))=e\phi (eSe (K))=e\phi (K)=e\phi e (K)=\psi(K)=0$, implies that $e\phi e\in l_S(SK)$. As $M$ is a centrally endo-AIP module, so there is a central element $\eta \in l_S(SK)$ such that $e\phi e \eta=e\phi e$. It is easy to see that $e\eta e \in l_T(K)$ and $\psi (e\eta e)= e\phi e (e\eta e)= (e\phi e\eta)e=e\phi e =\psi$. Now, it only remains to show that $e\eta e$ is a central element of $l_T(K)$. For it, let $\zeta \in l_T(K)$ be arbitrary. Then $\zeta (K)=0$ and for some $\theta \in S$, $\zeta=e\theta e$. Now $e\eta e \zeta =e\eta e (e\theta e)=e\eta (e\theta e)=e(e\theta e)\eta=e\theta e (e\eta)=e\theta e (e\eta e)=\zeta e\eta e $. Thus, $e\eta e$ is a central element of $l_T(K)$ such that for $\psi \in l_T(K)$, $\psi e\eta e=\psi$. Therefore, $l_T(K)$ is a centrally s-unital ideal of $T$. Hence, $N$ is a centrally endo-AIP module.
\end{proof}
Submodules of a centrally endo-AIP module need not be centrally endo-AIP. The following example illustrates it.
\begin{example}
	Let $M=\mathbb{Z}_p \oplus \mathbb{Q}$ ($p$ is any prime) be a $\mathbb{Z}$-module. By (Example 2.9, \cite{ORM}) $M$ is a reduced module. Since, $M$ is also a Rickart module, from Corollary \ref{RCE1.1} $M$ is centrally endo-AIP module. While, the submodule $N=\mathbb{Z}_p \oplus \mathbb{Z}$ of $M$ is not centrally endo-AIP. In fact, the submodule $N$ of $M$ is not endo-APP (see Example 4.4, \cite{DM})
\end{example}
Now, we discuss in the following proposition when a submodule of a centrally endo-AIP module is a centrally endo-AIP.
\begin{proposition}
	Let $N$ be a fully invariant submodule of a centrally endo-AIP module $M$ with $S=End_R(M)$ and $T=End_R(N)$. If every $\psi \in T$ can be extended to $\bar{\psi} \in S$, then $N$ is a centrally endo-AIP submodule.
\end{proposition}
\begin{proof}
	Let $X$ be a fully invariant submodule of $N$. As $X\trianglelefteq N$ and $N\trianglelefteq M$ implies $X\trianglelefteq M$. Suppose that $\phi \in l_T(X)$, then $\bar{\phi}(X)=\phi (X)=0$ which implies that $\bar{\phi}\in l_S(X)$. Since $M$ is a centrally endo-AIP module, there is a central element $\eta \in l_S(X)$ such that $\bar{\phi} \eta =\bar{\phi}$. Therefore, $\phi \eta |_N = \phi$ and $\eta |_N (X)=0$, which gives $\eta |_N \in l_T(X) $ as $N\trianglelefteq M$. Thus, $l_T(X)$ is a centrally s-unital ideal of $T$. Hence, $N$ is a centrally endo-AIP module.
\end{proof}

\begin{proposition}\label{P2.8}
	If every finitely generated $\mathbb{Z}$-module $M$ is centrally endo-AIP module then, $M$ is a torsion-free or semisimple module.
\end{proposition}

\begin{proof}
	Let $M$ be a finitely generated centrally endo-AIP $\mathbb{Z}$-module then, $M$ is an endo-APP $\mathbb{Z}$-module. So from (Proposition 4.8, \cite{DM}), $M$ is a semisimple or a torsionfree module. 
\end{proof}

The following example shows that the direct sum of centrally endo-AIP modules need not be a centrally endo-AIP.

\begin{example}\label{E2.16}
	The $\mathbb{Z}$-modules $\mathbb{Z}$ and $\mathbb{Z}_p$ (where $p$ is prime) both are centrally endo-AIP modules, while the direct sum $M=\mathbb{Z}\oplus \mathbb{Z} _p$ is not a centrally endo-AIP $\mathbb{Z}$-module. In fact, $M$ is neither semisimple nor torsionfree so by proposition \ref{P2.8}, $M$ is not a centrally endo-AIP module.
\end{example}
In the following proposition we discuss, when the direct sum of centrally endo-AIP modules is centrally endo-AIP.
\begin{proposition}\label{P2.17}
	Let $M=M_1\oplus M_2$, where $M_1$ and $M_2$ are centrally endo-AIP modules. If every $\phi\in Hom(M_i, M_j)$ (where $i\neq j\in \{1,2\}$) is a monomorphism, then $M$ is a centrally endo-AIP module.
\end{proposition}
\begin{proof}
	Assume that $N\trianglelefteq  M_1\oplus M_2$, then by (Lemma 1.10, \cite{RR}), $N=N_1\oplus N_2$, where $N_1 \trianglelefteq M_1$ and $N_2 \trianglelefteq M_2$. Now, let $S=End_R(M_1\oplus M_2)\cong \begin{pmatrix}
		T_1 && T_{12} \\ T_{21} && T_2 \end{pmatrix}$, where $T_i=End_R(M_i)$ for $i=1,2$ and $T_{ij}=Hom_R(M_j, M_i)$, for $i\neq j\in \{1,2\}$. Let $\alpha =\begin{pmatrix} \alpha _1 && \alpha_{12} \\ \alpha_{21} && \alpha_2 \end{pmatrix}$ and $\alpha \in l_S(N)$ then $\alpha(N)=0$ and for every $\beta_{ij}\in Hom_R(M_j, M_i)$, $\beta_{ij}\alpha_{ji}\in l_{T_i}(N_i)$, $i\neq j\in{1,2}$. Since $M_1$ and $M_2$ are centrally endo-AIP modules, there are some central elements $\phi _i \in l_{T_i}(N_i)$ for $i=1,2$ such that $ \beta_{ij}\alpha_{ji} \phi_i=\beta_{ij}\alpha_{ji}$, $i\neq j\in \{1,2\}$. Thus, for every $m\in M_i$, $(\beta_{ij}\alpha_{ji}\phi_i)(m)=(\beta_{ij}\alpha_{ji})(m)$, and so $\beta_{ij}((\alpha_{ji}\phi_i)(m)-(\alpha_{ji})(m))=0$. Therefore, by assumption $(\alpha_{ji}\phi_i )(m)-(\alpha_{ji})(m)=0$ for every $i\neq j \in \{1,2\}$ and $m\in M_i$, which implies that $\alpha_{ji}\phi _i= \alpha_{ji}$. Thus, for $\phi_1 \in T_1$ and $\phi_2\in T_2$, $\phi =\begin{pmatrix} \phi _1 && 0 \\ 0 && \phi _2 \end{pmatrix} \in l_S(N)$ and $\alpha \phi=\alpha $. Therefore, $l_S(N)$ is a centrally s-unital ideal of $S$. Hence, $M$ is a centrally endo-AIP module.
\end{proof}
\begin{theorem}\label{T2.18}
	Let $M=\bigoplus_{\lambda \in \Lambda} M_\lambda$, where $M_\lambda$ is a centrally endo-AIP module for every $\lambda \in \Lambda$ and $M_\lambda \cong M_\nu$ for every $\lambda, \nu \in \Lambda$. Then $M$ is a centrally endo-AIP module.
\end{theorem}
\begin{proof}
	First, we prove the theorem for $\Lambda=\{1,2,...,n\}$. Now, we assume ${(M_\lambda)}_R \cong X_R$ for every $\lambda \in \Lambda$. Let $N$ be a fully invariant submodule of $M$, then by (Lemma 1.10, \cite{RR}), $N=\bigoplus_{\lambda=1}^n N_\lambda$, where $N_\lambda=N\cap M_\lambda \subseteq X$ and each $N_\lambda$ is fully invariant in $X$. It is observe that, if ${S'}=End_R(X)$ then $S\cong Mat_n ({S'})$ and
	$l_S(N)=\begin{pmatrix} l_{S'}(N_1) && l_{S'}(N_2) && ...&& l_{S'}(N_n)\\ l_{S'}(N_1) && l_{S'}(N_2) &&...  && l_{S'}(N_n)\\ ... && \:... && ...&& ...\\ l_{S'}(N_1) && l_{S'}(N_2) && ...&& l_{S'}(N_n) \end{pmatrix}$. So, if $\phi \in l_S(N)$ then $\phi=\begin{pmatrix} \phi_1 && \phi_2 && ... && \phi_n \\ \phi_1 && \phi_2 && ...&& \phi_n\\ ...&&...&&...&&\\ \phi_1 && \phi_2 && ...&& \phi_n \end{pmatrix} $, where $\phi_\lambda \in l_{S'}(N_\lambda)$ for $\lambda\in \Lambda$. Since $X$ is a centrally endo-AIP module, so for each $\phi_{\lambda} \in l_{S'}(N_{\lambda})$, there are some central elements $\psi_{\lambda} \in l_{S'}(N_\lambda)$ such that $\phi_\lambda \psi_\lambda=\phi_\lambda$. Hence, $\phi = \begin{pmatrix}
		\phi_1 && \phi_2 && ...&& \phi_n\\ \phi_1 && \phi_2 &&...&& \phi_n\\ ...&&...&&...&&...\\ \phi_1 && \phi_2 &&...&&\phi_n\end{pmatrix} =\begin{pmatrix}
		\phi_1\psi_1 && \phi_2 \psi_2 &&...&& \phi_n\psi_n\\ \phi_1\psi_1 && \phi_2 \psi_2 &&...&& \phi_n\psi_n\\ ...&&...&&...&&...\\ \phi_1\psi_1 && \phi_2 \psi_2 &&...&& \phi_n\psi_n \end{pmatrix}$ which implies that $\phi=\begin{pmatrix} \phi_1 && \phi_2 && ...&& \phi_n\\ \phi_1 && \phi_2 &&...&& \phi_n\\ ...&&...&&...&&...\\ \phi_1 && \phi_2 &&...&&\phi_n \end{pmatrix}\begin{pmatrix} \psi_1 && 0 && ...&& 0\\ 0 && \psi_2 &&...&& 0\\ ...&&...&&...&&...\\ 0 && 0 &&...&& \psi_n \end{pmatrix}$. Since, for each $\lambda \in \Lambda$, $\psi_\lambda \in l_{S'}(N_\lambda)$ is a central element, therefore $\psi =\begin{pmatrix} \psi_1 && 0 && ...&& 0\\ 0 && \psi_2 &&...&& 0\\ ...&&...&&...&&...\\ 0 && 0 &&...&& \psi_n \end{pmatrix}$ is a central element of $l_{S}(N)$.
	Now $\psi (N)=0$ because $\psi_\lambda\in l_{S'}(N_\lambda)$ for every $\lambda\in \Lambda$. Therefore, $l_S(N)$ is a centrally s-unital  ideal of $S$. By assuming $\Lambda$ an infinite set, the proof can be extended to a column finite matrix ring. Hence, $M$ is a centrally endo-AIP module.
\end{proof}

\begin{theorem}\label{Free1.1}
	The following statements are equivalent:
	\begin{enumerate}
		\item[(i)] Every projective right $R$-module is centrally endo-AIP module;
		\item[(ii)] Every free $R$-module is centrally endo-AIP module;
		\item[(iii)] $R$ is a centrally AIP ring.
	\end{enumerate}
\end{theorem}
\begin{proof}
	$(i)\Rightarrow (ii)$ It is clear.\\
	$(ii)\Rightarrow (i)$ It is well known that a projective module is a direct summand of a free module. Since, from (ii) every free module is centrally endo-AIP module. Therefore, from Proposition \ref{DS1.1} every projective module is centrally endo-AIP module.\\
	$(ii)\Rightarrow (iii)$ It is clear that $R_R$ is a free right $R$-module. So, by (ii) $R_R$ is a centrally endo-AIP $R$-module. Therefore, $R$ is a centrally AIP ring.\\
	$(iii)\Rightarrow (ii)$ Let $M=R^{(\Lambda)}$ be a free $R$-module and $\Lambda$ be an arbitrary index set. Since $R$ is a centrally AIP ring, so from Theorem \ref{T2.18} $M$ is a centrally endo-AIP module.
\end{proof}
\begin{remark}
	From Theorem \ref{Free1.1} it is clear that, if $R$ is a centrally AIP ring then polynomial ring $R[x]$ and matrix ring $M_n(R)$ are centrally AIP rings, see also (Lemma 3.4, \cite{AAK}) and (Proposition 3.14, \cite{AAK}). 
\end{remark}
\section{Endomorphism rings of centrally endo-AIP modules}
In this section we study the endomorphism ring of centrally endo-AIP modules and discuss the equivalency among quasi-Baer modules, centrally endo-AIP modules and endo-AIP modules.
\begin{proposition}\label{P3.1}
	The endomorphism ring of centrally endo-AIP module is centrally AIP ring.
\end{proposition} 
\begin{proof}
	Let $S=End_R(M)$ be an endomorphism ring of $M$, $T$ be an ideal of $S$ and $f \in l_S(T)$. Then $f(T(M))=0$ which implies that $f\in l_S(T(M))$. It is clear that $T(M)$ is a fully invariant submodule of $M$.  Since $M$ is a centrally endo-AIP module, there exists a central element $g\in l_S(T(M))$ such that $fg=f$. Thus, $gT(M)=0 \Rightarrow gT=0$ which implies that $g\in l_S(T)$. Hence, $S$ is a centrally AIP ring.
\end{proof}

\begin{corollary}\label{SP1.1}
	The endomorphism ring of a centrally endo-AIP module is a semiprime ring.	
\end{corollary}
\begin{proof}
	Since, from (Proposition 2.9, \cite{AAK}) every centrally AIP ring is semiprime ring. Therefore, the proof follows from Proposition \ref{P3.1}.
\end{proof}
\begin{remark}
	We observe that when we take the class of finitely generated projective module $M$ over a centrally AIP ring $R$, then the endomorphism ring of $M$ is a centrally AIP. In particular, the centrally AIP property is Morita invariant (Theorem 3.5, \cite{AAK}).
\end{remark}
The following example shows that the converse of proposition \ref{P3.1} need not be true in general.
\begin{example}\label{E3.2}
	Let $M=\mathbb{Z}_{p^{\infty}}$ be a $\mathbb{Z}$-module, where $p$ is prime . It is well known that $End_R(M)\cong \mathbb{Z}_{(\bar{p})}$ (ring of p-adic integers) (see Example 3, page 216 \cite{LF}), which is a commutative domain and endo-AIP ring (see Example 3.2, \cite{DM}). Therefore, it is a centrally AIP ring. Also, $M$ is not an endo-AIP module (see Example 3.2, \cite{DM}). Hence, by Proposition \ref{CEE1.1} $M$ is not a centrally endo-AIP module.
\end{example}
Recall that a module $M$ is locally principally quasi-retractable module, if for every principal ideal $I$ of $S=End_R(M)$ such that $r_M(I)\neq 0$, then there exists a non-zero endomorphism $\psi \in S$ such that $r_M(I)=\psi (M)$ (see Definition 3.3, \cite{DM}).
\begin{proposition}\label{LPQ1.1}
	Let $M$ be a locally principally quasi-retractable module. If $S=End_R(M)$ is a left centrally AIP ring, then $M$ is a centrally endo-AIP module.
\end{proposition}
\begin{proof}
	Let $N$ be a fully invariant submodule of $M$. Then for every $f \in l_S(N)$, $Sf S\subseteq l_S(N)$. Thus, $0\neq N\subseteq r_M(Sf S)$. Since $M$ is a locally principally quasi-retractable module and $r_M(Sf S)\neq 0$, so there exists $0\neq g \in S$ such that $N\subseteq r_M(Sf S)=g (M)$ and $f \in l_S(Sg S)$. Since $S$ is a left centrally AIP ring, there is a central element $h \in l_S(Sg S)$ such that $f h =f=h f$. Now as $N\subseteq g (M)$ so $h (N)\subseteq h (g (M))=0$. Therefore, $h \in l_S(N)$. Hence, $M$ is centrally endo-AIP module. 
\end{proof}
Recall that a right $R$-module $M$ has uniform dimension $n$ (written as $u.dim(M_R)=n$) if there is an essential submodule $N$ of $M$, which is a direct sum of $n$ uniform submodules. If no such integer exists, then $u.dim(M)=\infty$. Further, a ring $R$ has finite right (left) uniform dimesion, if $u.dim(R_R)=n$ ($u.dim(_RR)=n$) for a positive integer $n$.
\begin{proposition}\label{UD1.1}
	Let $M$ be a centrally endo-AIP module and $S=End_R(M)$. If $S$ has finite right uniform dimension, then $S$ is a quasi-Baer ring.
\end{proposition}
\begin{proof}
	Let $M$ be a centrally endo-AIP module. Then, from Proposition \ref{P3.1} $S$ is centrally AIP ring. So, by assumption $S$ is a centrally AIP ring with finite right uniform dimension. Hence, from (Theorem 5.1, \cite{AAK}) $S$ is a quasi-Baer ring.
\end{proof}
A right $R$-module $M$ is called semi-projective \cite{Vedadi}, if $T=Hom_R(M, TM)$ for any cyclic right ideal $T$ of $S=End_R(M)$.
\begin{corollary}
	Endomorphism ring of a centrally endo-AIP semi-projective retractable module with finite uniform dimension is quasi-Baer. 
\end{corollary}
\begin{proof}
	Let $M$ be a semi-projective retractable module with finite uniform dimension. Then, from (Theorem 2.6, \cite{Vedadi}) $S=End_R(M)$ has finite right uniform dimension. Therefore, from Proposition \ref{UD1.1}, $S$ is a quasi-Baer ring.
\end{proof}
\begin{corollary}\label{CP1.1}
	Let $M$ be an $R$-module with endomorphism ring $S=End_R(M)$. If $M$ is a centrally endo-AIP module and $u.dim(S_S)=1$, then $S$ is a prime ring.
\end{corollary}
\begin{proof}
	Let $M$ be a centrally endo-AIP module and $u.dim(S_S)=1$. Then from Corollary \ref{SP1.1} and Proposition \ref{UD1.1}, $S$ is a semiprime quasi-Baer ring. If $S$ is not a prime ring, then $l_S(S\phi)=Se$ for some $0\neq \phi \in S$ and $e\in \mathbb{S}_r(S)$. Since $S$ is semiprime quasi-Baer ring, $e\in S$ is central. Thus $S=Se\oplus S(1-e)$, a contradiction as $u.dim(S_S)=1$. Hence, $S$ is a prime ring. 
\end{proof}
\begin{proposition}
	Let $M$ be an endo-AIP module and $S=End_R(M)$. If $S$ is local ring then $S$ is prime.
\end{proposition}
\begin{proof}
	It is clear from (Theorem 3.1, \cite{DM}) that the endomorphism ring of an endo-AIP module is AIP ring. Thus, $S$ is a local AIP ring. Since, every local AIP ring is a prime ring (Proposition 5.3, \cite{AAK}). Hence, $S$ is a prime ring.
\end{proof}
Recall from \cite{Primeness}, a fully invariant submodule $N\trianglelefteq M$ is said to be a prime submodule of $M$ (in this case $N$ is said to be prime in $M$), if for any ideal $T$ of $S=End_R(M)$, and for any fully invariant submodule $N'\trianglelefteq M$, $T(N')\subset N$ implies $T(M)\subset N$ or $N'\subset N$. Further, a fully invariant submodule $N$ of $M$ is called semiprime submodule if it is an intersection of prime submodules of $M$. A right $R$-module $M$ is called prime module if $\{0\}$ is prime in $M$ while $M$ is called semiprime module if $\{0\}$ is semiprime submodule of $M$.
\begin{proposition}
	Let $M$ be a semiprime (prime) module and $S=End_R(M)$ satisfies ascending chain condition on its principal left ideals. Then, the following conditions are equivalent:
	\begin{enumerate}
		\item[(i)] $M$ is a quasi-Baer module;
		\item[(ii)] $M$ is an endo-AIP module;
		\item[(iii)] $M$ is a centrally endo-AIP module.
	\end{enumerate}
\end{proposition}
\begin{proof}
	$(i)\Rightarrow (ii)$ It is clear.\\
	$(ii)\Rightarrow (iii)$ Let $M$ be an endo-AIP module, $N$ be a fully invariant submodule of $M$ and $S$ satisfies ascending chain condition on its principal left ideals. Then, from (Proposition 3.9, \cite{DM}) $M$ is a quasi-Baer module. So, $l_S(N)=Se$ for some $e^2=e\in S$. Since $M$ is a semiprime module, so by (Theorem 2.9, \cite{Primeness}) $S$ is a semiprime ring. Thus, from (Proposition 1.17, \cite{Birken}), $\mathbb{S}_r(S)=\mathbb{S}_l(S)$. Hence, $M$ is a centrally endo-AIP module.\\
	$(iii)\Rightarrow (i)$ Since, $M$ is a centrally endo-AIP module and $S$ satisfies ascending chain condition on principal left ideal. As every centrally endo-AIP module is endo-AIP, so from (Proposition 3.9, \cite{DM}) $M$ is a quasi-Baer module.
\end{proof}


\begin{thebibliography}{0}
	
		\bibitem{ORM}
	N. Agayev, S. Halıcıoğlu and A. Harmancı, ``On Rickart modules", {\small\it Bulletin of Iranian Mathematical Society} {\small\bf 38(2)} (2012), 433-445.
	
		\bibitem{Nazim}
	N. Agayev, T. Özen and A. Harmanci, ``On a class of semicommutative modules", {\small\it Proceedings-Mathematical Sciences} {\small\bf 19(2)} ( 2009), 149-158.
	
	\bibitem{Vedadi}
	A. Haghany and M. R. Vedadi, ``Study of semi-projective retractable modules", {\small\it Algebra Colloquium}, {\small\bf 14(03)} (2007) 489-496. 
		\bibitem{Birken}
	G. F. Birkenmeier, J. Y. Kim, and J. K. Park, ``Principally quasi-Baer rings", {\small\it Communications in Algebra} {\small\bf 28(2)} (2001), 639-660.
	
		\bibitem{PM} 
	P. M. Cohn, ``On the free product of associative rings", {\small\it Mathematische Zeitschrift}, {\small\bf 71(1)} (1959), 380-398.
	
		\bibitem{DM}
	P. A. Dana and A. Moussavi, ``Modules in which the annihilator of a fully invariant submodule is pure", {\small\it Communications in Algebra}, {\small\bf 48(11)} (2020), 4875-4888. 
	
		\bibitem{LF}
	L. Fuchs, ``Infinite Abelian Group I", {\small\it Pure and Applied Mathematics Series}, New York-London : Academic press {\small\bf Vol 36} (1970).
	
		\bibitem{DJ}
	D. J. Fieldhouse, ``Pure Theories", {\small\it Math. Ann.} {\small\bf 184(1)} (1969), 1-18.
	
		\bibitem{IK}
	I. Kaplansky, ``Rings of Operators", {\small\it Mathematics Lecture Note Series} New York: W.A. Benjamin (1965).
	
		\bibitem{GL}
	G. Lee, S. T. Rizvi and C. S. Roman, ``Rickart modules", {\small\it Communications in Algebra} {\small\bf 38(11)} (2010), 4005-4027.
	
		\bibitem{ZR}
	Z. Liu and R. Zhao, ``A generalization of PP-rings and p.q.-Baer rings", {\small\it Glasgow Math. Journal}, {\small\bf 48(2)} (2006), 217-229.
	
		\bibitem{QBT}
	Q. Liu, B.Y. Ouyang and T.S. Wu, ``Principally Quasi-Baer Modules", {\small\it Journal of Mathematical Research and Exposition}, {\small\bf 29} (2009), 823-830.
	
		\bibitem{AAK}
	A. Majidinya, A. Moussavi and K. Paykan, ``Rings in which the annihilator of an ideal is pure", {\small\it Algebra Colloquium} {\small\bf 22(1)} (2015), 917-968.
	
		\bibitem{CE}
	C. E. Rickart, ``Banach algebras with an adjoint operation", {\small\it Annals of Mathematics } {\small\bf 47(3)} (1946), 528-550.
	
		\bibitem{RR} 
	S. T. Rizvi and C. S. Roman, ``Baer and quasi-Baer module", {\small\it Communications in Algebra} {\small\bf 32(1)} (2004), 103-123.
	
		\bibitem{Primeness}
	N. V. Sanh, N. Anh Vu, K. F. U. Ahmed, S. Asawasamrit and L. P. Thao, ``Primeness in module category", {\small\it Asian-European Journal of Mathematics}  {\small\bf 3(1)} (2010), 145-154.
		
\end{thebibliography}
\end{document}